\documentclass[12pt]{amsart}

\usepackage{enumerate,url,amssymb,  nicefrac, mathrsfs, graphicx, pdfsync}

\newtheorem{theorem}{Theorem}[section]
\newtheorem{lemma}[theorem]{Lemma}
\newtheorem*{lemma*}{Lemma}

\theoremstyle{definition}
\newtheorem{definition}[theorem]{Definition}
\newtheorem{example}[theorem]{Example}
\newtheorem{conjecture}[theorem]{Conjecture}

\theoremstyle{remark}
\newtheorem{remark}[theorem]{Remark}

\numberwithin{equation}{section}

\newcommand{\abs}[1]{\lvert#1\rvert}

\newcommand{\A}{\mathbb{A}}
\newcommand{\C}{\mathbb{C}}

\newcommand{\Q}{\mathbb{Q}}

\newcommand{\W}{\mathscr{W}}

\newcommand{\R}{\mathbb{R}}
\newcommand{\X}{\mathbb{X}}

\newcommand{\Y}{\mathbb{Y}}

\newcommand{\dtext}{\textnormal d}

\newcommand{\onto}{\xrightarrow[]{{}_{\!\!\textnormal{onto\,\,}\!\!}}}

\newcommand{\bydef}{\stackrel {\textnormal{def}}{=\!\!=} }
\DeclareMathOperator{\diam}{diam}

\DeclareMathOperator{\dist}{dist}

\DeclareMathOperator{\loc}{loc}

\DeclareMathOperator{\Id}{Id}
\def\le{\leqslant}
\def\ge{\geqslant}

\begin{document}

\title{A Neohookean model of plates}

\author[T. Iwaniec]{Tadeusz Iwaniec}
\address{Department of Mathematics, Syracuse University, Syracuse,
NY 13244, USA}
\email{tiwaniec@syr.edu}

\author[J. Onninen]{Jani Onninen}
\address{Department of Mathematics, Syracuse University, Syracuse,
NY 13244, USA and  Department of Mathematics and Statistics, P.O.Box 35 (MaD) FI-40014 University of Jyv\"askyl\"a, Finland}
\email{jkonnine@syr.edu}

\author[P. Pankka]{Pekka Pankka}
\address{Department of Mathematics and Statistics, P.O. Box 68 (Gustaf H\"allstr\"omin katu 2b), FI-00014 University of Helsinki, Finland}
\email{pekka.pankka@helsinki.fi}

\author[T. Radice]{Teresa Radice}
\address{Universit\'a degli Studi di Napoli ``Federico II'', Dipartimento di Matematica e Applicazioni ``R. Caccioppoli''
Via Cintia, 80126 Napoli, Italy}
\email{teresa.radice@unina.it}

\thanks{ T. Iwaniec was supported by the NSF grant DMS-1802107.
J. Onninen was supported by the NSF grant DMS-1700274. P. Pankka was  supported by the Academy of Finland project \#297258.}

%    General info

\date{\today}


\maketitle

\begin{abstract}  This article is about hyperelastic deformations of plates (planar domains) which minimize a neohookean type energy. Particularly, we investigate  a stored energy functional introduced by J.M. Ball  in his seminal paper ``\textit{Global invertibility of Sobolev functions and the interpenetration of matter}". The mappings under consideration are Sobolev homeomorphisms and their weak limits. They are monotone in the sense of C. B. Morrey. One major advantage of adopting monotone Sobolev mappings lies in the existence of the energy-minimal deformations. However, injectivity is inevitably lost, so an obvious question to ask is:  what are the largest subsets of the reference configuration on which minimal deformations remain injective? The fact that such subsets have full measure should be compared with the notion of \textit{global invertibility} which deals with subsets of the deformed configuration instead. In this connection we present a Cantor type construction to show that both the branch set and its image may have positive area. Another novelty of our approach lies in allowing the elastic deformations be free along the boundary, known as \textit{frictionless problems.}
\end{abstract}

%\tableofcontents

\section{Introduction}
We study hyperelastic deformations of neohookean materials in planar domains called plates.  These problems are motivated by recent remarkable relations between Geometric Function Theory (GFT)~\cite{AIMb, HKb,  IMb, Reb}  and the theory of Nonlinear Elasticity (NE)~\cite{Anb, Bac, Cib}. Both theories are governed by variational principles. Here  we confine ourselves to deformations of bounded Lipschitz domains $\X,\Y \subset \mathbb R^2 \simeq \mathbb C\,$ of finite connectivity. The general theory of hyperelasticy deals  with Sobolev homeomorphisms $h \colon \X \onto \Y$ of nonnegative Jacobian determinant, $J_h \bydef \textrm{det} Dh \geqslant 0$, which minimize a given  stored energy functional
\[\mathcal E [h] \bydef \int_\X E\left(\abs{Dh} , \det Dh\right)\, \dtext x ,\, \textnormal{ where }   E \colon \R_+ \times \R_+ \to \R.\]
The stored energy function  $\,E \colon \R_+ \times \R_+ \to \R\,$, is determined  by the elastic and mechanical properties of the material.

Here the $\,2\times 2\,$-matrix $Dh\in \R^{2 \times 2}$ is referred to as the deformation gradient and $\abs{Dh}$ denotes its Hilbert-Schmidt norm. We are largely concerned with orientation-preserving homeomorphisms $h \colon \X \onto \Y$ of the Sobolev class $\W^{1,p} (\X, \C)$, denoted by $\mathscr H^{p} (\X, \Y)$, as well as their weak and strong limits.
If $p\ge 2$ then every $h\in \mathscr H^{p} (\X, \Y)$ extends  continuously up to the boundary, still denoted by $h \colon \overline{\X }\to \overline{\Y}$, see~\cite{IObound}.

The term neohookean refers to a stored energy function $E$ which increases to infinity when $J_h$ approaches zero. The neohookean materials have gain a lot of interest  in mathematical models of nonlinear elasticity~\cite{Ba1, BPO1, CL, Ev, SiSp}. A model example takes the form
\begin{equation}
\mathsf E_q^p [h] = \int_\X \left[ \abs{Dh}^p+\frac{1}{ (\det Dh)^q} \right]\, \dtext x \, ,  \quad p >1 \textnormal{ and } q>0\, .
\end{equation}
Throughout this paper we tacitly assume that the class of admissible homeomorphisms is nonempty; that is, there is $h \in \mathscr H^{p} (\X, \Y ) $ such that $\mathsf E_q^p [h] < \infty$. In particular, $\,\mathbb X\,$ and $\,\mathbb Y\,$ are of the same topological type.
As a first step toward understanding the existence problems we shall accept  the weak limits of energy-minimizing sequences of homeomorphisms as legitimate deformations. Thus we allow so-called \textit{weak interpenetration of matter}; precisely, squeezing of the material can occur. This changes the nature of minimization problem to the extent that the minimal energy (usually attained) can be strictly smaller than the infimum energy among homeomorphisms.

In a seminal work of J. M. Ball~\cite{Ba0}   injectivity properties were studied for  pure displacement problems. That is, the admissible deformations are specified a priori on the entire boundary of the reference configuration $\X$. More specifically,  choose and fix $\varphi  \in \mathscr H^{p} (\X, \Y),\, p \geqslant 2$, and introduce the following class of admissible deformations:
\[\mathcal A^p \bydef \{h \in \mathscr C (\overline{\X}, \C) \cap \W^{1,p} (\X, \C) \colon J(x,h)>0\;  \textnormal{ a.e.}, \;  h=\varphi \textnormal{ on } \partial \X\}\]

J. M. Ball~\cite{Ba0} proved the following:
\begin{itemize}
\item  If $p>2$ and $q>0$, then there exist the energy-minimal map $h_\circ \in \mathcal A^p$ such that
\[ \mathsf E_q^p [h_\circ]= \underset{h\in \mathcal A^p}{\inf} \mathsf E_q^p [h]  \, . \]
\item For every  $h\in \mathcal A^p $ with $p>2$ the following set has full measure:
\begin{equation}\label{eq:ballinjae}
{\mathbb Y}_{h} \bydef  \{y \in h(\overline{\X})  \colon h^{-1} (y) \textnormal{ is a single point }\} \subset \overline{\mathbb Y }\end{equation}
\end{itemize}
This result has been referred to as {\it global invertibility} for two reasons. First, because $\mathbb Y_{h}$ has full measure in $\overline{\mathbb Y}$. Second,  because any minimizer  $\,h_\circ\,$ upon restriction to $\,h_\circ^{-1}(\mathbb Y_{h_\circ})\,$ becomes injective. For further generalizations of the global invertibility result when $p \ge 2$ we refer to~\cite{FG}.

P.G. Ciarlet and J. Ne\v cas~\cite{CN} studied mixed boundary value problems (the displacement is prescribed only on a portion of the boundary of the reference configuration $\X$). In their mixed problems the pure displacement condition $\,h = \varphi\,$  on $\,\partial \mathbb X\,$ is replaced by 
\begin{equation}\label{CNcondition}\int_\X \det Dh (x)\, \dtext x \le \abs{h(\X)}  \, . 
\end{equation}
 They showed that the minimizers of $\mathsf E_q^p$, $ p > 2$,  subject to such class of deformations are globally invertible.

Usually, in GFT  the boundary values of homeomorphisms $h$ are not given. For example extremal Teichm\"uller quasiconformal mappings are not prescribed on the boundary; the boundary does not even exist for compact Riemann surfaces. In NE this is interpreted as saying that the elastic  deformations are allowed to slip along the boundary,  known as {\it frictionless problems}~\cite{Bac, Ba11, Cib, CN}. 

  Our goal is to enlarge the class of homeomorphisms (as little as possible) to ensure the existence of minimizers in that class.  The right way is to adopt  the {\it monotone Sobolev mappings}~\cite{IOmono}.  Indeed, that this class is a bare minimal enlargement of homeomorphisms follows from  a Sobolev variant of   the classical Youngs' approximation theorem.  Its classical topological setting  asserts that a  continuous map between compact oriented topological $2$-manifolds (such as plates and  thin films) is monotone  if and only if it is a uniform limit of homeomorphisms. Monotonicity, the concept of C.B. Morrey~\cite{Mor}, simply means that for a continuous $h \colon \overline{\X} \to \overline{\Y}$ the preimage $h^{-1} (C)$ of a continuum  $\,C \subset\overline{\Y}$ is a continuum in $\overline{\X}$.  The above-mentioned Sobolev variant reads as: 
\begin{theorem}[Approximation by Sobolev Homeomorphisms] \label{thmmono}~\cite{IOmono}
Let $\X$ and $\Y$ be bounded Lipschitz planar domains. Suppose that  $h \colon \overline{\X} \onto \overline{\Y}$ is a monotone Sobolev mapping in $\mathscr W^{1,p} (\X, \mathbb R^2)$, $1< p <\infty$. Then $h$  can be approximated  in norm topology of  $\mathscr W^{1,p} (\X, \mathbb R^2)$ (and uniformly) by homeomorphisms  $h_j\, \colon \X \onto \Y$.
\end{theorem}
Let us introduce the notation   $\mathscr M^p\,(\overline{\mathbb X} , \overline{\mathbb Y})$  for  the class of orientation preserving monotone mappings $h \colon \overline{\mathbb X} \onto \overline{\mathbb Y}$ in $\W^{1,p} (\X, \C)$.
Our first result guarantees the existence of a minimizer of neohookean energy  among Sobolev monotone deformations.
\begin{theorem}\label{thm:monoexistence}
Let $p\ge 2$ and $q>0$. Then there exists $h_\circ \in \mathscr M^p (\overline{\X}, \overline{\Y})$ such that
\begin{equation} \label{MonotoneMinimizer}\mathsf E_q^p [h_\circ]= \underset{h\in \mathscr M^p (\overline{\X}, \overline{\Y})}{\inf} \mathsf E_q^p [h]  \, . \end{equation}
\end{theorem}
All the evidence points to the following:
\begin{conjecture}\label{con:main}
Every minimizer $h_\circ \in \mathscr M^p (\overline{\X}, \overline{\Y})$ in (\ref{MonotoneMinimizer})   is a homeomorphism.
\end{conjecture}

For example, this conjecture is confirmed when $\X$ and $\Y$ are circular annuli, see~\cite{IOne}. Also, the conjecture is valid if $\X=\Y$, in which case it is  relatively easy to see that the identity map minimizes the energy, see Example~\ref{ex:id}. However, it is not known whether  the identity map minimizes the neohookean energy $\mathsf E_q^p$ when $p<2$  (compressible neohookean materials).  It is worth noting that this is not the case for the $p$-harmonic energy, see~\eqref{eq:fbn=2p>1}.

We give an affirmative answer to these questions for  neohookean materials whose associated energy integrand grows sufficiently fast. Precisely, we have

\begin{theorem}\label{thm:homeexistence}
Let $p>2$ and $q>1$ such that $\frac{2}{p}+\frac{1}{q} \le 1$. Then there exists a homeomorphism  $h_\circ \in \mathscr H^p ({\X}, {\Y})$ such that
\[ \mathsf E_q^p [h_\circ]= \underset{h\in \mathscr M^p (\overline{\X}, \overline{\Y})}{\inf} \mathsf E_q^p [h]  =  \underset{h\in \mathscr H^p ({\X}, {\Y})}{\inf} \mathsf E_q^p [h]\, . \]
\end{theorem}
The existence of  monotone  minimizer $h_\circ$ is ensured by  Theorem~\ref{thm:monoexistence}, and the fact that $h_\circ$ is a homeomorphism follows from the next result.
\begin{theorem}\label{thm:monodiscrete}
Let $p>2$ and $q \geqslant \frac{p}{p-2}$. If $h \in  \mathscr M^p (\overline{\X}, \overline{\Y})$ and 
\begin{equation}\label{LocallyFiniteEnergy}
\int_{\Omega} \left(|Dh|^p + \frac{1}{J_h^q} \right) \dtext x \,<\,\infty \,, \,\,  \textnormal{\textit{for every compact subset}}  \; \Omega \Subset\mathbb X\, ,\end{equation}
then $h\colon \X \onto \Y$ is a homeomorphism.
\end{theorem}
\begin{remark}\label{rem:16}
Theorem \ref{thm:monodiscrete} also holds for $p=2$ and $q= \infty$, in which case the locally finite energy condition (\ref{LocallyFiniteEnergy}) should be stated as
\[ \int_{\mathbb X} |Dh|^2 \, \dtext x \, < \, \infty \,\, \,\textnormal{ and } \,\, \, \frac{1}{\det Dh} \in \mathscr L^{\infty}(\mathbb X). \]
That is,  $\det Dh(x) \geqslant \frac{1}{C} >0$ with a constant $C= \| J_h^{-1}\|_{\mathscr L^{\infty}(\mathbb X)}$.
\end{remark}

Theorem~\ref{thm:monodiscrete} is sharp, namely it fails if  $0< q< \frac{p}{p-2}$, as the following example shows.
 \begin{example}\label{ex:nonhomeo}
 For $0< q< \frac{p}{p-2}$ there exists a non-injective $h \in  \mathscr M^p (\overline{\X}, \overline{\Y})$ with $ \mathsf E_q^p [h] <\infty$.
\end{example}

This example raises a question about  partial injectivity  of  $h \in  \mathscr M^p (\overline{\X}, \overline{\Y})$ with $ \mathsf E_q^p [h] <\infty$ when  $0<q< \frac{p}{p-2}$. First, we have, 

\begin{theorem}\label{thm:partialinj}
Suppose that a monotone map $h \colon \overline{\X} \onto \overline{\Y}$ of Sobolev class $\W^{1,2} (\X, \mathbb C)$ has positive Jacobian determinant. Then
\begin{itemize}
\item $h$ is globally invertible in the sense of (\ref{eq:ballinjae})
\item In addition, there exists $\X_h $ of full measure in $\X$ such that $h$ restricted to $\X_h$ is injective.  
\end{itemize}
\end{theorem}

Next in lines is the study of the {\it branch set}
\[\mathcal B_h \bydef \{x \in \X \colon  h \textnormal{ fails to be homeomorphic near } x\}  \]
and its image $h(\mathcal B_h)$.  Recall that  for the Dirichlet  energy the branch set of the energy-minimal mapping $h\in \mathscr M^{2} (\overline{\mathbb X} , \overline{\mathbb Y}) $ may have a positive area, whereas $h (\mathcal B_{h}) \subset \partial \Y$   (actually a nonconvex part of $\partial \Y$)~\cite{CIKO, IOan}. We show, however, that  under the assumptions  of Example~\ref{ex:nonhomeo}  both the branch set and also the image of the branch set may have a positive area. Recall that if $q\ge  \frac{p}{p-2}>0$, then any monotone map $h$ with  $ \mathsf E_q^p [h] <\infty$ is injective  by Theorem~\ref{thm:homeexistence}.
\begin{example}\label{ex:branch}
If $0< q< \frac{p}{p-2}$, then there exists $h \in  \mathscr M^p (\overline{\X}, \overline{\Y})$ with $ \mathsf E_q^p [h] <\infty$ such that $\abs{\mathcal B_h}>0$ and $\abs{h(\mathcal B_h)}>0$.
\end{example}
Our example is based on a Cantor type construction, see Section~\ref{Construction} for the construction.

Returning to Conjecture~\ref{con:main},  with the aid of the complex partial derivatives, $h_z= \frac{\partial h}{\partial z}$ and $h_{\bar z}= \frac{\partial h}{\partial \bar z}$, we express the energy as
\begin{equation}\label{eq:neop=2}
\mathsf E^2_1[h] = \int_\X \left[2 \big(\abs{h_z}^2+ \abs{h_{\bar z}}^2 \big) + \big(\abs{h_z}^2- \abs{h_{\bar z}}^2 \big)^{-1} \right]\, \dtext z \, .
\end{equation}
Clearly, one cannot perform outer variations $h_\varepsilon =h+\,\varepsilon \,\eta$, with  $\eta \in \mathscr C^\infty_\circ (\X)$ as they live out the class of monotone Sobolev mappings $\mathscr M^p (\overline{\X}, \overline{\Y})$. Thus we loss the Euler--Lagrange equation, which is the major source of difficulty here. Such a difficulty is widely recognized in the theory of Nonlinear Elasticity.  This forces us to  rely on the inner variation of the independent variable $z_\varepsilon = z + \varepsilon \tau (z)$, where $\tau \in \mathscr C_\circ^{\infty}(\mathbb X, \mathbb R^2)$. The inner variational equation for the minimizer takes the form
\begin{equation}\label{eq:innerp}
\frac{\partial}{\partial {z}} \left[ \left( 1-\frac{p}{2}\right)|Dh|^p + \frac{1+q}{J_h^q}\right] \,= \, 2p\, \frac{\partial}{\partial \overline{z}} \left[|Dh|^{p-2} \overline{h_z}h_{\overline{z}}\right]
\end{equation}
see formula (2.6), page 648 in \cite{IKOre}. 

Here, the complex partial derivatives $\frac{\partial}{\partial z}$ and $\frac{\partial}{\partial \bar z}$ are understood in the sense of distributions.  For $p=2$ and $q=1$, this simplifies as,
\begin{equation}\label{eq:inner}
2 \frac{\partial}{\partial \bar z} \left[ h_z \overline{h_{\bar z}}\right] =  \frac{\partial}{\partial  z} \left[ \frac{1}{\det Dh}\right]\, ,   \qquad \textnormal{with } \;\mathsf E^2_1[h]< \infty \, .
\end{equation}
Let us name~\eqref{eq:innerp} and (\ref{eq:inner}) {\it neohookean Hopf systems}.

It is worth noting that monotone Lipschitz solutions to the neohookean Hopf system~\eqref{eq:inner} are homeomorphisms.  In this connection we recall that for  the Dirichlet energy the inner-stationary solutions  are  always Lipschitz continuous, see~\cite{IKOre}.  Actually,  a solution of~\eqref{eq:inner} in $\mathscr M^4 (\overline{\X}, \overline{\Y})$ will turn out to be a homeomorphism. This follows from the next result,  simply by taking $p=2$ and $q=1$.
\begin{theorem}\label{thm:hopfhomeo}
 Consider a monotone mapping $h \colon \overline{\mathbb X} \onto \overline{\mathbb Y}$ of finite neohookean energy:
 \[ \mathsf E_q^p[h]  =  \int_{\mathbb X} \left( |Dh|^p + J_h^{-q}\right) \dtext x, \quad  p>1 \quad  \textnormal{and} \quad  q> 0 \,  . \]
 Assume that $h \in \mathscr W^{1,s}_{\loc} (\mathbb X, \mathbb R^2)$, for some $s \geqslant  \frac{p}{q}+2$ and $s>p$,  satisfies the equation ~\eqref{eq:innerp}. Then $h$ is a homeomorphism of $\mathbb X$ onto $\mathbb Y$. 
 \end{theorem}

\section{Preliminaries}
\subsection{Monotone in the sense of Lebesgue}
There are several notions commonly known in literature as monotonicity. To avoid confusion we use the term {monotone in the sense of Lebesgue} for one of these.  This notion goes back to
H. Lebesgue~\cite{Le} in 1907.

\begin{definition} \label{def:monoleb}
Let $\X$ be an open subset of $\R^2$. A continuous mapping $h\colon {\X} \to \R^2$ is  \emph{monotone in the sense of Lebesgue} if for every open set $\Omega \subset {\X}$ we have
\begin{equation}\label{diamdef}
\diam h(\overline{\Omega}) = \diam h(\partial \Omega).
\end{equation}
Note that for a real-valued function ~\eqref{diamdef} can be stated as
\begin{equation}\label{a}
\min_{\overline{\Omega}}h=\min_{\partial \Omega} h   \qquad   \textnormal{(minimum principle)}
\end{equation}
\begin{equation}\label{b}
 \max_{\overline{\Omega}} h = \max_{\partial \Omega}h \qquad   \textnormal{(maximum principle)}.
\end{equation}
\end{definition}

\subsection{Modulus of continuity and conformal energy}
In the next lemma, $\X $ and $ \Y $ are $\ell$-connected Lipschitz domains in $\R^2$, see~\cite[Lemma 4.3]{IOjems}% By the  conformal change of variable one can, without loss of generality, to assume that the reference configuration is a Jordan domain.
\begin{lemma}\label{lem:modofcontdirc}
To every pair $(\X, \Y)$ of $\ell$-connected bounded Lipschirtz domains, $\ell \ge 2$, there corresponds a constant $C=C(\X, \Y)$ such that for $h \in \mathscr H^2 (\X, \Y)$ and $\ell \ge 2$ we have
\begin{equation}\label{eq:modofcontdirc}
\abs{h(x_1)-h(x_2)}^2 \le \frac{C\cdot  \int_\X \abs{Dh}^2}{\log \left( 1+ \frac{\diam \X}{\abs{x_1-x_2}}\right)} \, , \qquad x_1,x_2 \in \overline{\X}
\end{equation}
whenever $h \in  \mathscr H^2 (\X, \Y)$ and $x_1 \not= x_2$ in $\overline{\X}$.
\end{lemma}
\begin{remark}\label{rem:modofcontdirc}
Inequality~\eqref{eq:modofcontdirc} fails
when $\ell=1$ and $p=2$. For this, consider a sequence of the M\"obius transformations $h_k \colon \mathbb D \onto \mathbb D$, $k=1,2, \dots$
\[h_k(z) = \frac{z+a_k}{1+a_k z}\, , \qquad 0<a_k<1  \quad \textnormal{and } a_k \nearrow 0 \, .  \]
The mappings are fixed at two boundary points,
\[h_k(1)=1 \quad \textnormal{ and } \quad \quad h_k (-1)=-1 \,  \]
and are equiintegrable:
 \[\int_\X \abs{Dh_k}= 2\, \int_\X J_h (x) \, \dtext x = 2\pi \, . \]
The sequence $h_k \colon  \mathbb D \onto \mathbb D$ approaches the constant mapping $h(z) \equiv 1$.
Obviously, we are losing equicontinuity of the boundary mappings $h_k \colon \partial \mathbb D \onto \partial \mathbb D$, in contradiction with~\eqref{eq:modofcontdirc}.
\end{remark}

\subsection{Change of variables formula}
We say that $h \colon \X \to \mathbb C$ satisfies the {\it Lusin (N) condition} if for every $E \subset \X$ such that $\abs{E}=0$ we have $\abs{h(E)}=0$.
\begin{lemma}\label{lem:lusin}
Suppose that $h \in \W^{1,2}(\X , \C)$ with $J_h >0$ a.e. Then $h$ satisfies the Lusin (N) condition.
\end{lemma}
Lemma~\ref{lem:lusin} follows because a monotone mapping in the sense of Lebesgue in the Sobolev class $\W_{\loc}^{1,2}(\X, \mathbb C)$
satisfies the Lusin (N) condition, see e.g.~\cite[Lemma 1.2]{HMa}.  On the other hand, a mapping  $h \in \W^{1,2}(\X , \C)$ with $J_h >0$ a.e.  is
monotone in the sense of Lebesgue, see~\cite[Proposition 4.1]{IKOin}. The Lusin property is very important as it allows us to obtain the change of variables formula, see~\cite[Theorem 6.3.2]{IMb}.
\begin{lemma}\label{lem:changeofvariables}
Let $h \colon \X \to \mathbb C$ be a mapping in the Sobolev class $\W^{1,2} (\X, \C)$ with $J_h(x)>0$ for almost every  $x$ in $\X$.
If $\eta$ is a nonnegative Borel measurable function on $\mathbb C$
and $A$  a Borel measurable set in $\X$, then
\begin{equation}
\int_A \eta \big( h(x)\big)\, J_h (x) \, \dtext x = \int_{h(A)} \eta (y) N_h (y, A)\, \dtext y
\end{equation}
where $N_h(y,A)$ denotes the cardinality of the set $\{x\in A \colon h(x)=y\}$.
\end{lemma}

\subsection{Weak compactness of Jacobians}

\begin{lemma}\label{lem:weakcomp}
Let $\X$ be a domain in $\C$ and  $h_k \in \W^{1,2} (\X, \C)$  a sequence of mappings with $J(x,h_k) \ge 0$ a.e. in $\X$ converging weakly
in $\W^{1,2} (\X, \C)$ to $h \in \W^{1,2} (\X, \C)$. Then the Jacobians $J(x,h_k)$ converge weakly
in $\mathscr L^1_{\loc} (\X)$ to $J(x,h) $ and $J(x,h) \ge 0$ a.e. in $\X$. Precisely,
\[ \displaystyle{\lim_{k \rightarrow \infty}} \int_{\mathbb X} \varphi(x) J(x, h_k)\,  \dtext x = \int_{\mathbb X} \varphi (x)J(x, h) \,  \dtext x\]
for every $\varphi \in \mathscr L^{\infty}(\mathbb X)$ with compact support.
\end{lemma}
For a proof of this lemma we refer to~\cite[Theorem 8.4.2]{IMb}.

\subsection{Polyconvexity of neohookean integrand}
The remarkable feature of the Neohookian energy is the polyconvexity of its integrand. Instead of the general definition~\cite{Bac, Morpc} let us confine ourselves, as a consequence, to the so-called {\it gradient inequalities}.

Let $p \ge 2$ and $q>0$.  For arbitrary square matrices $A\in \mathbb R^{2\times 2}$ and $A_\circ\in \mathbb R^{2\times 2}$, we have
\[
\begin{split}
\abs{A}^p- \abs{A_\circ}^p &= \left( \abs{A}^2 \right)^\frac{p}{2} - \left( \abs{A_\circ}^2 \right)^\frac{p}{2}  \ge \frac{p}{2} \left( \abs{A_\circ}^2 \right)^{\frac{p}{2}-1} \left( \abs{A}^2 - \abs{A_\circ}^2   \right) \\
& \ge  \frac{p}{2}\,  \abs{A_\circ}^{p-2} \, 2 \langle A_\circ , A-A_\circ \rangle =  \langle   \,  p\,  \abs{A_\circ}^{p-2} A_\circ  , A-A_\circ \, \rangle
\end{split}
\]
where $\langle \cdot , \cdot \rangle$ stands for the scalar product of matrices.

For arbitrary positive numbers $J$ and $J_\circ$, we have
\[\frac{1}{J^q}-\frac{1}{J^q_\circ}  \ge \frac{q}{J_\circ^{q+1}}  \left( J_\circ -J\right) \, .\]
Next we show that the lower-semicontinuity of neohookean integral follows from the above gradient inequalities.

\begin{lemma}\label{lem:poly}
Let $\X$ be a domain in $\C$, $p\ge 2$ and $q>0$. Suppose that  $h_k \in \W^{1,p} (\X, \C)$  is a sequence of mappings with $J(x,h_k) >0 $ a.e. in $\X$ converging weakly
in $\W^{1,p} (\X, \C)$ to $h \in \W^{1,p} (\X, \C)$ and $\mathsf E^p_q [h_k]\le E<\infty$. Then
\[  \mathsf E^p_q [h] \le \liminf_{k\to \infty}  \mathsf E^p_q [h_k]  \, .  \]
\end{lemma}
\begin{proof}
Choose and fix a positive number $\varepsilon$ and a compact subset $F\Subset \X$, the above gradient inequalities imply
\[
\begin{split}
& \int_F \left[  \abs{Dh_k(x)}^p + [\varepsilon + J(x,h_k)]^{-q}  \right]  - \int_F \left[  \abs{Dh(x)}^p + [\varepsilon + J(x,h)]^{-q}  \right] \\
&\quad  \ge p \int_F  \langle   \,  p\,  \abs{Dh(x)}^{p-2} Dh(x)  , Dh_k(x)-Dh(x) \, \rangle \, \dtext x \\ & \quad  + q \int_F \frac{J(h,x)-J(h_k, x)}{[\varepsilon + J(x,h)]^{q+1}} \, \dtext x
\end{split}
\]
Now letting $k \to \infty$ the first integral term goes to zero, because $Dh_k -Dh \to 0$ weakly in $\mathscr L^p (F)$ whereas $\abs{Dh}^{p-2} Dh$ belongs to the dual space of  $\mathscr L^p (F)$. Concerning the last integral term we appeal to Lemma~\ref{lem:weakcomp} on weak compactness of the Jacobian determinants. Accordingly,
\[  \int_F \frac{J(h,x)-J(h_k, x)}{[\varepsilon + J(x,h)]^{q+1}} \, \dtext x \to 0  \]
where our test function $\varphi (x) = \frac{\chi_F(x)}{[\varepsilon + J(x,h)]^{q+1}  }\le \frac{1}{\varepsilon}$ lies in $\mathscr L^\infty (\X)$ and has compact support. We thus have an estimate
\[
\begin{split}
&  \int_F \left[  \abs{Dh(x)}^p + [\varepsilon + J(x,h)]^{-q}  \right] \, \dtext x \\
&  \le \liminf_{k \to \infty} \int_F \left[  \abs{Dh_k(x)}^p + [\varepsilon + J(x,h_k)]^{-q}  \right]  \, \dtext x \\
&  \le \liminf_{k \to \infty} \int_\X \left[  \abs{Dh_k(x)}^p + [J(x,h_k)]^{-q}  \right]  \, \dtext x= \liminf_{k \to \infty}  \mathsf E^p_q [h_k] < \infty \, .
\end{split}
\]
Consider a sequence $\{\varepsilon_j\}$ of positive numbers converging to zero and an increasing sequence of compact subsets $F_1 \subset F_2 \subset \dots$ with $\bigcup F_n=\X$. Thus,
\[  \int_{F_n} \left[  \abs{Dh(x)}^p + [\varepsilon_n + J(x,h)]^{-q}  \right] \, \dtext x \le  \liminf_{k \to \infty}  \mathsf E^p_q [h_k] < \infty \, .  \]
Letting $n \to \infty$, by the monotone convergence theorem, the desired estimate $\mathsf E^p_q [h] \le \liminf\limits_{k\to \infty}  \mathsf E^p_q [h_k] $ follows.
\end{proof}

\section{The case of $\X=\Y$}
When $\X=\Y$ the identity map is a natural candidate for the minimizer. The case $1<p<2$ (compressible neohookean materials), however, offers further challenges. To explain this we take $q=1$ for simplicity. First of all when $p\ge 2$ we have
\begin{example}\label{ex:id}
The identity mapping $h_{\circ}=\Id \colon \X \onto \X$ minimizes the neohookean energy $\mathsf E^p_1$,  when $p\ge 2$, subject to all homeomorphisms in $\mathscr H^{p} (\X, \X)$. In fact, this follows from the inequality
\begin{equation}\label{eq:identity}
\mathsf E^p_1[h] \ge (2^\frac{p}{2}+1)\abs{\X} = \mathsf E^p_1[h_{\circ}]   \qquad \textnormal{for all } h \in \mathscr H^{p} (\X, \X) \, .
\end{equation}
\end{example}
The proof of this inequality is obtained by the methods of   {\it Free-Lagrangians}.
 A free-Lagrangian is a special case of {\it null Lagrangian}~\cite{Bac}. This is a nonlinear differential $2$-form
 defined on Sobolev homeomorphisms $h \colon \X \onto \Y$ whose integral depends only on the homotopy class of $h$, see~\cite{IOan}. The simplest free Lagrangians is the area  form $\det Dh(x)\, \dtext x$  for $h \in \mathscr W^{1,p}(\X, \X)$ with $p \geqslant 2$. This is  a key player in the proof of~\eqref{eq:identity}. The unavailability of the area form is exactly why our  arguments for Theorem
 \ref{thm:monodiscrete}  do not apply when $p<2$. Nevertheless, it is not clear whether ~\eqref{eq:identity} remains valid for $p<2$. In~\cite{KOpharm} it is shown that  the identity mapping $h^{\ast}=\Id \colon \mathbb D_\circ \onto \mathbb D_\circ$ from the punctured disk $\mathbb D_\circ = \{z \in \C \colon 0< \abs{z} < 1\}$ onto itself does not minimize the $p$-harmonic energy when $1 \leqslant p < p_1$, for some $1<p_1<2$. Namely,
\begin{equation}\label{eq:fbn=2p>1}
\underset{h\in \mathscr H^{p} (\mathbb D_\circ, \mathbb D_\circ)}{\inf} \int_\A \abs{Dh(x)}^p\, \dtext x <  \int_\A \abs{ D h^{\ast} (x) }^p\, \dtext x \,   .
\end{equation}
Let us point out that the identity mapping is always a minimizer in the class of radially symmetric homeomorphisms.

\begin{proof}[Proof of  Example~\ref{ex:id}]

%\subsection{Proof of  Example~\ref{ex:id}}\label{3.2}

%\begin{proof}[Proof of Inequality~\eqref{eq:identity}]
 First applying Young's inequality $ab \le \frac{a^{\alpha}}{\alpha}+ \frac{b^{\beta}}{\beta}$, $\frac{1}{\alpha}+\frac{1}{\beta}=1$, we observe a pointwise inequality
\[\abs{Dh}^p  \ge p\,  2^{\frac{p-4}{2}}\,  \abs{Dh}^2 -(p-2) 2^{\frac{p-2}{2}} \, .\]
Equality occurs if $\abs{Dh}^2=2$.
Then, Hadamard's inequality $\abs{Dh}^2 \ge 2 J_h$, $J_h=\det Dh$, yields
\[\abs{Dh}^p  \ge p\,  2^{\frac{p-2}{2}}\, \left[  J_h -1+\frac{2}{p} \right] \, .\]
Again, we have the equality when $h=\Id$. Hence
\[\abs{Dh}^p  + \frac{1}{J_h} \ge \left(p\, 2^\frac{p-2}{2} -1\right)J_h -(p-2)2^\frac{p-2}{2} +J_h + \frac{1}{J_h} \, \]
where  $J_h + \frac{1}{J_h} \ge 2$. Therefore,
\[\abs{Dh}^p  + \frac{1}{J_h} \ge \left(p\, 2^\frac{p-2}{2} -1\right)J_h -(p-2)2^\frac{p-2}{2} +2 \, .\]
This gives us a desired estimate of the stored energy integrand by means of free Lagrangians; namely, $J_h$ and a constant function.
Integrating over the domain $\X=\Y$, the claimed estimate~\eqref{eq:identity} follows. Equality occurs for the identity map; and only for isometries $h \colon \X \onto \X$.

\end{proof}

\section{Proof of Theorem~\ref{thm:partialinj}}
\begin{proof}
{\bf Step 1. ($\abs{\Y_h}=\abs{\Y}$)}. First, since $h\in \W^{1,2} (\X, \C)$ and $J_h(x) >0$ a.e., Lemma~\ref{lem:changeofvariables} gives
\begin{equation}\label{eq:partial1form}
\int_\X  J_h (x) \, \dtext x = \int_{\Y}  N_h (y, \X)\, \dtext y
\end{equation}
where $N_h(y,\X)$ denotes the cardinality of the set $\{x\in \X \colon h(x)=y\}$.

Second, for an orientation preserving homeomorphism $g \colon \X \onto \Y$ in the Sobolev class $\W^{1,2} (\X, \C)$ we have
\begin{equation}
\int_\X  J_g (x) \, \dtext x = \abs{\Y}\, .
\end{equation}
Now combining this with  Theorem~\ref{thmmono} and Lemma~\ref{lem:weakcomp} for an orientation preserving  monotone $h \colon \X \onto \Y$ in $\W^{1,2} (\X, \C)$ we have
\begin{equation}\label{eq:partialform2}
\int_\X  J_h (x) \, \dtext x = \abs{\Y}\, .
\end{equation}
Therefore, by~\eqref{eq:partial1form} and~\eqref{eq:partialform2} for a monotone $h \colon \X \onto \Y$ in $\W^{1,2} (\X, \C)$ with $J_h(x)>0$ a.e. in $\X$,
we obtain  $ N_h (y, \X) =1 $ for a.e. $y$ in $\Y$; that is, $\abs{\Y_h}=\abs{\Y}$. Since  $\Y$ is a Lipschitz domain, it holds that $\abs{\partial \Y}=0$.\\
{\bf Step 2. ($\abs{\X_h}=\abs{\X}$)}. The claim is that $\abs{h^{-1} (\Y_h)}=\abs{\X}$. Indeed,  according to Lemma~\ref{lem:changeofvariables}, we have
\[ \int_{\X \setminus h^{-1} (\Y_h)} J_h(x) = \abs{\Y \setminus \Y_h} =0\, .  \]
Furthermore since $J_h(x)>0$ a.e. in $\X$ we have $\abs{\X \setminus h^{-1} (\Y_h)} =0$.
\end{proof}

\section{Constructing   Examples~\ref{ex:nonhomeo}}% and~\ref{ex:id}}

\begin{proof}[Proof  of  Example~\ref{ex:nonhomeo}]
%\end{proof}
%\subsection{Proof  of  Example~\ref{ex:nonhomeo}}
Consider the rectangles $ \X = (-1,1) \times (-2,2)= \Y$. To construct a monotone map $h \colon \overline{\X} \onto \overline{\Y}$ we choose and fix parameters $a>- \frac{1}{p}$; $b>1-\frac{1}{p}$ such that $a+b < \frac{1}{q}$. This choice is possible because $0<q< \frac{p}{p-2}$. The map in question is defined by the rule \[ h(x,y) = (u(x,y), v(x,y)) \,\,\,\,  \textnormal{where} \]
\[u(x,y)=x|x|^a, \,\,\,\,\, \textnormal{for} \, -1 \leqslant x \leqslant 1.\]
\begin{equation*} v(x,y)=
\begin{cases}
y |x|^b, \,\,\, \textnormal{for} \, |y| \leqslant 1 \\
(2-|x|^b)y + 2(|x|^b-1) \frac{y}{|y|}, \,\, \textnormal{for} \, 1 \leqslant |y| \leqslant 2
\end{cases}
\end{equation*}
It is worth noting that for $x$ fixed the function $y \to v(x,y)$ is linear in each of the following intervals $y \in [-1,1]$, $y \in [1,2]$ and $y \in [-2,-1]$, see Figure \ref{fig2}. Clearly, we have 
\[ h^{-1}(0,0)= \mathbb I  \bydef \{(0,y)\colon  |y| \leqslant 1\} \]
Outside this interval $h$ is a bijection $h \colon \overline{\X} \setminus \mathbb{I} \onto \overline{Y} \setminus \{ (0,0)\}$. 
Its inverse $h^{-1} \bydef  \overline{\mathbb Y} \setminus \{ (0,0)\} \onto \overline{\X} \setminus \mathbb I$, takes the form $f(u,v)= (x(u,v), y(u,v))$, where
\[ x(u,v)= u|u|^{\frac{-a}{1+a}}, \,\,\, -1 \leqslant u \leqslant 1 \]
\[y(u,v)=
\begin{cases}
\displaystyle{\frac{v\pm 2 (1- |u|^{\frac{b}{1+a}})}{2- |u|^{\frac{b}{1+a}}}}, \,\,\,\, &\textnormal{whenever} \,\, \pm v \geqslant |u|^{\frac{b}{1+a}} \,\, \textnormal{respectively} \\
\displaystyle{v|u|^{\frac{-b}{1+a}}}, \,\,\,\,\,\,\,\,\,\,\,\,\,\,\,\,\,\,\,\,\, &\textnormal{whenever} \,\,\, |v| \leqslant |u|^{\frac{b}{1+a}}.
\end{cases}
\]
Thereby, $h$ is a monotone map. 

Concerning the energy of $h$, because of symmetries it is enough to evaluate the energy over the rectangle $(0,1) \times (0,2)$. The formula takes the form
\[ \mathsf E^p_q[h] = 4 \int_0^1 \left[\int_0^1 \mathsf E(x,y) \dtext y + \int_1^2\mathsf E(x,y) \dtext y \right] \dtext x\]
where
\[ \mathsf E(x,y)= |Dh(x,y)|^p + \frac{1}{J_h(x,y)^q} \,\,\, \textnormal{for} \,\,x, y \geqslant 0. \]
Consider two cases:
\\
\underline{Case 1}. $0 \leqslant y \leqslant 1$, so
\[\begin{cases}
u(x,y)= x^{a+1}, \,\, &\textnormal{so} \,\, u_x=(a+1)x^a \,\, \textnormal{and} \,\, u_y=0 \\
v(x,y)=x^b y, \,\, &\textnormal{so} \,\, v_x=bx^{b-1} y \,\, \textnormal{and} \,\, v_y=x^b.
 \end{cases}
 \]
 Hence $|Dh(x,y)|^p \leqslant C x^{ap} + x^{(b-1)p}$ and $J_h(x,y)=u_xv_y-u_yv_x=u_xv_y=(a+1)x^{a+b}$. Since $ap>-1$ and $(b-1)p>-1$ we have \[
 \int_0^1 \int_0^1 |Dh(x,y)|^p \dtext x \, \dtext y\, < \infty
 \]
 On the other hand $J_h(x,y)^q=(a+1)^qx^{(a+b)q}$. Since $a+b< \frac{1}{q}$, we have
 \[ \int_0^1 \int_0^1 \frac{\dtext x \, \, \dtext y}{J_h(x,y)^q} < \infty\]
In conclusion, the energy over the square $0 \leqslant x \leqslant 1$, $0 \leqslant y \leqslant 1$ is finite.\\
 \underline{Case 2}.  $1 \leqslant y \leqslant 2$, so
 \[\begin{cases}
u(x,y)= x^{a+1}, \,\, &\textnormal{so} \,\, u_x=(a+1)x^a \,\, \textnormal{and} \,\, u_y=0 \\
v(x,y)=2y- x^b y+ 2x^b -2, \,\, &\textnormal{so} \,\, v_x=bx^{b-1} (2-y) \,\, \textnormal{and} \,\, v_y=2-x^b.
 \end{cases}
 \]
 Hence
 \[ |Dh(x,y)|^p \leqslant C x^{ap} + x^{(b-1)p}\]
 \[ J_h(x,y) =u_xv_y -u_yv_x =u_xv_y= (a+1)x^a(2-x^b) \geqslant (a+1)x^a\]
 As in Case 1,
 \[ \int_0^1 \int_0^1 |Dh(x,y)|^p \dtext x \,\dtext y < \infty \]
On the other hand
\[ \int_0^1 \int_0^1 \frac{\dtext x \, \dtext y}{J_h(x,y)^q} \leqslant \frac{1}{(a+1)^q} \int_0^1 \int_0^1 \frac{\dtext x \, \dtext y}{x^{aq}} < \infty\]
because $aq \leqslant (a+b)q < 1$. In conclusion $\mathsf E_q^p[h] < \infty$, as desired.
 %$h \colon \overline{\X} \onto \overline{\Y}$ by
%\begin{equation}
%h(x,y) = \begin{cases} (x \abs{x}^a , y \abs{x}^b) \quad & \textnormal{ for } \abs{y} \le 1 \\
%(x \abs{x}^a, \frac{y}{\abs{y}}?????) \quad & \textnormal{ for } 1 \le \abs{y} \le 2 \, . \end{cases}
%\end{equation}
%Then $h^{-1} (0,0) = \{0\} \times [-1,1] \bydef {\mathbb I}$. Furthermore,  $h \colon \overline{\X} \setminus {\mathbb I} \to \overline{\Y} \setminus \{0\}$ is a bijection. {\bf need to be finished}

\begin{figure}[h]
\centering
\includegraphics[width=12cm]{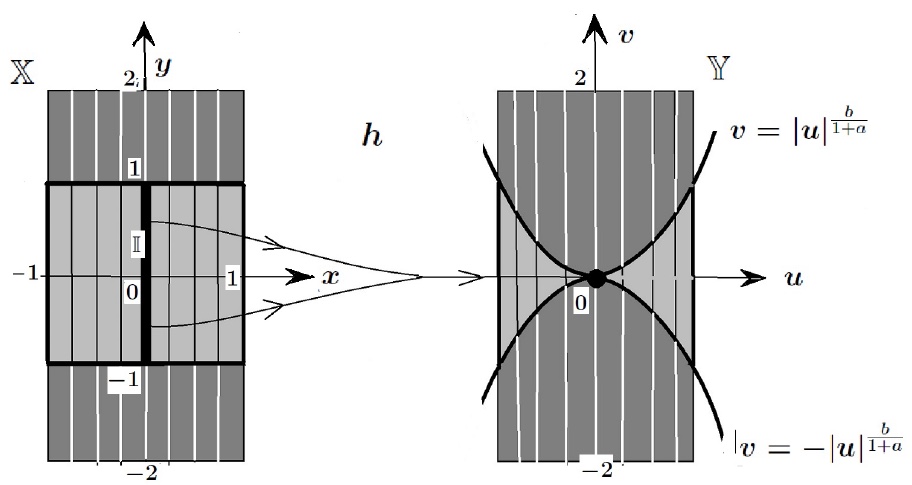}
\caption{}\label{fig2}
\label{fig:cups}
\end{figure}

\end{proof}

\subsection{An extension} \label{rem:preparing} We just have constructed a monotone map $h \colon  \overline{\mathbb X }\onto \overline{\mathbb Y}$ of finite energy  which equals the identity on the vertical sides of the rectangle $\,\overline{\mathbb X} = [-1, 1] \times [-2, 2]\,$. However, restricted to the horizontal sides it is not the identity; it takes the form:
\[ h (x,2) =  (x|x|^a , 2)\;\;\; \textnormal{and} \;\;\; h(x,-2) = (x|x|^a  , -2) ,\;\textnormal{for} \; -1\leqslant x \leqslant 1 \, . \]
We shall still need a map that is equal to the identity on the entire boundary. For this purpose we extend $\,h\,$  to a map $\, \widetilde{h} : \widetilde{\mathbb X}\onto \widetilde{\mathbb Y}\,$,  where $\,\widetilde{\mathbb X} = \widetilde{\mathbb Y } = [-1, 1] \times[-3, 3]\,$, by the rule,

\begin{equation} \widetilde{h}(x,y)=
\begin{cases}
\Big(x|x|^a (3-y) \,+\, x (y-2)\;,\; y \Big), \,\,\, & \textnormal{when} \, 2 \leqslant y \leqslant 3 \\
h(x,y) \;,\;\; &  \textnormal{when}\; -2\leqslant y \leqslant 2\\
\Big(x|x|^a (3+y) \,-\, x (y+2)\;,\; y \Big), \,\,\, &  \textnormal{when} \, -3 \leqslant y \leqslant -2 \, . 
\end{cases}
\end{equation}
 Clearly $\,\widetilde{h}\,$ is monotone and equal to the identity on $\,\partial \widetilde{\mathbb X}\,$. Just as in the computation above  we see  that $\,\mathsf E^p_q[\widetilde{h}] < \infty\,$. Proceeding further in this direction, we may extend $\,\widetilde{h}\,$ to the square $\,\mathbf S \bydef [-4, 4] \times [-4, 4]\,$ by letting it be equal to the identity outside $\,\widetilde{\mathbb X}\,$.
 Let us record this in the following lemma.

  \begin{lemma}\label{ModelMapping}
  For every $\, p>2\,$  and $\,0< q< \frac{p}{p-2}$ there exists a non-injective monotone map  $\,\Phi\in  \mathscr M^p (\mathbf S, \mathbf S)$  of finite $ \mathsf E_q^p\,$-energy, which is the identity near the boundary of $\,\mathbf S\,$. Precisely, we have the following average energy
  \begin{equation}\label{ModelInequality}
  \mathsf E_q^p [\Phi] \,= \frac{1}{|\mathbf S|} \int_\mathbf S \left[ \abs{D\Phi(y)}^p+\frac{1}{ [\,\det D\Phi(y)]^q} \right] \dtext y\;\;\bydef \mathbf E \,< \infty
  \end{equation}
  where $\,|\mathbf S | = 16\,$ is the area of the square  $\mathbf S =[-4, 4] \times [-4, 4]$.
 \end{lemma}
 \subsubsection{Rescaling} Formula (\ref{ModelInequality}) can be rescaled to an arbitrary square $\,Q\subset \mathbb R^2\,$ in place of $\,\mathbf S\,$. Let us  discuss it in a somewhat greater context. For, choose and fix a prototype energy integral over a square $\,\mathbf S \subset \mathbb R^2\,$ centered at the origin and of side-length $\,L\,$,
 \begin{equation}\label{PrototypeEnergy}
 \mathscr E[\Phi] \bydef \frac{1}{|\mathbf S|} \int_\mathbf S \textrm{E}( D\Phi(y)) \,\textnormal dy \;,
 \end{equation}
 This integral is assumed to exist for some adequate mappings\;$\,\Phi : \mathbf S \onto \mathbf S \,$\,,\,$\Phi(0) = 0\,$.
 Note that the stored-energy integrand depends solely on the deformation gradient $\,D\Phi\,$. Now take any square $\,Q\,$ centered at $\,a\in \mathbb R^2 \,$ and of side-length $\ell$. Then the mapping
 \begin{equation}\label{RescalledMapping}
 h_Q : Q \onto Q\;, \; \textnormal{defined by} \;\; h_Q(x) \bydef  a + \frac{\ell}{L} \,\Phi\left(\frac{L}{\ell} (x-a)\right)
 \end{equation}
has the same average energy as $\Phi$,
 \begin{equation}\label{RescaledEquation}
 \mathscr E[h_Q] = \frac{1}{|\mathbf Q|} \int_\mathbf Q \textrm{E}( Dh_Q(x)) \,\textnormal dx  \; =\; \mathscr E[\Phi]
 \end{equation}
This is an obvious  consequence of the chain rule  $\,Dh_Q(x) = D\Phi(y)\,$, where $\,y = \frac{L}{\ell} (x-a)\,$ is a variable used as a substitution in the integral (\ref{PrototypeEnergy}).
For later use, it should be noted that if $\,\Phi\,$ is monotone, so is $\,h_Q\,$.  Also, if $\,\Phi\,$ is the identity map near $\,\partial \mathbf S\,$ then so is $\,h_Q\,$ near $\,\partial Q\,$.

\section{Cantor Type Construction of  Example~\ref{ex:branch}}\label{Construction}
\subsection{Construction of Cantor set}
We shall work with  closed squares whose sides are parallel to the standard coordinate axes of $\R^2$, but most of the definitions and formulas will be coordinate-free.
\subsubsection{Cornersquares} Suppose we are given a  square $\,Q \subset \mathbb R^2\,$ and a parameter $\,0 < \varepsilon < 1\,$. Write it as $\,Q =  I \times J\,$, where $\,I , J \, \subset \mathbb R\,$ are closed intervals of the same length $\,\ell = |I| = |J|\,$. These might be called respectively the horizontal and the vertical factors of $\,Q\,$. The notation $\,\varepsilon I\,$ and $\,\varepsilon J\,$ will stand for the intervals of the  same centers but  $\,\varepsilon\,$-times smaller in length, respectively. Cutting them out from  $\,I\,$ and $\,J\,$ gives the decompositions:
$$
I \setminus \varepsilon I =  I_- \cup I_+  \;\;\textnormal{and}\;\; J\setminus \varepsilon J = J^- \cup J^+\,
 $$
into the left and the right, as well as into the lower and the upper subintervals. Note that we suppressed the explicit dependence on $\,\varepsilon\,$ in the notation. This parameter will be determined later during our induction procedure. Now the Cartesian product consists of four sub-squares:
$$
(I \setminus \varepsilon I) \times (  J\setminus \varepsilon J)   \;=\;   Q^+_+\,\cup\, Q^+_-\,\cup\, Q^-_-\,\cup\, Q^-_+\, . 
 $$
 Explicitely, we have the formulas:
 $$
 Q^+_+\,\bydef I_+ \times J^+ \,,\;  Q^+_-\,\bydef I_- \times J^+ \,,\,  Q^-_-\, \bydef I_- \times J^- \,,\,  Q^-_+ \bydef I_+ \times J^-\, . 
 $$
Each of these sub-squares touches exactly one corner of $\,Q\,$, which motivates our calling $\,Q^+_+\,,\, Q^+_-\,,\, Q^-_-\,,\, Q^-_+\, \,$ the \textit{cornersquares} of $\,Q\,$; more precisely, the first generation of cornersquares. We shall also spot the so-called \textit{centersquare} of $\,Q\,$,  defined by $\,\varepsilon Q =  \varepsilon I \times \varepsilon J\,$, see the left hand side of Figure \ref{CornSquares}.
 \subsubsection{Second generation of cornersquares} Choose another positive $\,\varepsilon\,$-parameter, say  $\,\varepsilon = \varepsilon_2 \,$. Then every cornersquare of $\,Q\,$ gives rise to its own four cornersquares  determined by this parameter, see the middle part  of Figure \ref{CornSquares}.  In this way we obtain sixteen cornersquares of so-called second generation. According to our notation these are:

$$ \mbox{ $ Q^{+ +}_{++}\;\;Q^{++}_{+-}\;\; Q^{+-}_{+-}\;\; Q^{+-}_{++}\;\;   $} $$
$$ \mbox{$ Q^{+ +}_{- +}\;\; Q^{++}_{--}\;\; Q^{+-}_{--}\;\; Q^{+-}_{-+}\;\; $} $$
$$ \mbox{$ Q^{-+}_{-+}\;\; Q^{-+}_{--}\;\; Q^{--}_{--}\;\; Q^{--}_{-+}\;\;$} $$
$$ \mbox{$ Q^{-+}_{++}\;\; Q^{-+}_{+-}\;\; Q^{--}_{+-}\;\; Q^{--}_{++}\;\; $} $$

See also the third generation of 64 cornersquares in the right hand side of  Figure \ref{CornSquares}.

 \begin{figure}[h]
\centering
\includegraphics[width=12.5cm]{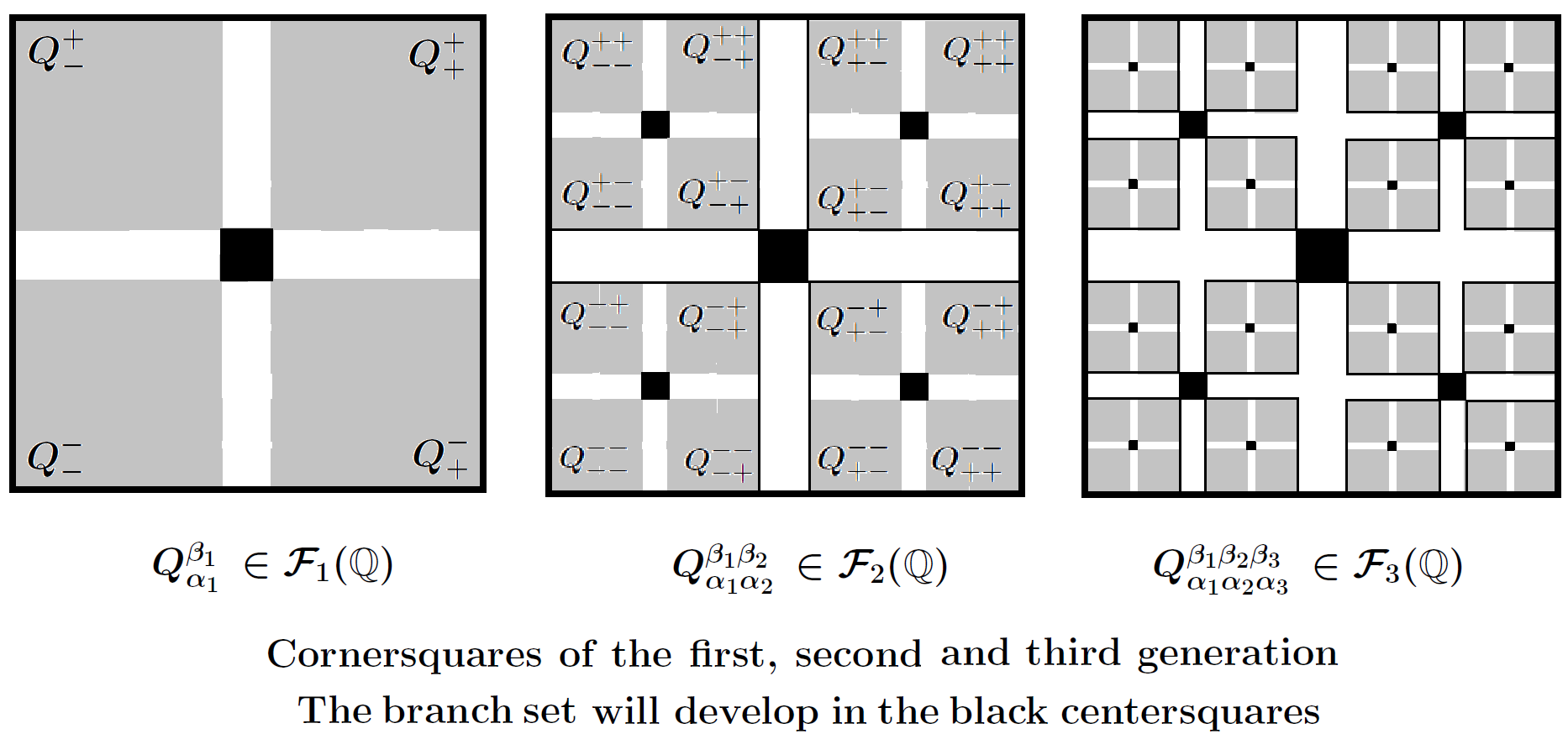}
\caption{Cornersquares as building blocs for a Cantor type construction}\label{CornSquares}
\end{figure}

\subsubsection{The induction procedure} Fix a sequence of $\varepsilon\,$-parameters rapidly decreasing to 0,  say $\,(\varepsilon_1, \varepsilon_2, ... )\,$ with $\varepsilon_n = 4^{-n}\,$. We begin with the base $\,1\times 1\,$ square $\,\mathbb Q\subset \mathbb R^2\,\,$ and the first $\,\varepsilon\,$- parameter equal to $\, \varepsilon_1\,$. This gives us the  first generation of four cornersquares $\,Q^{\beta_1}_{\alpha_1} \subset \mathbb Q\,$,  where both indices run over the set $\,\{+\,,\,-\}\,$. We let $\,\mathcal F_1\,$ denote this family of cornersquares.

In the second step we take $\varepsilon_2\,$ as the $\,\varepsilon\,$-parameter and look at the cornersquares of every $\,Q^{\beta_1}_{\alpha_1}\,$. Denote them  by $\,Q^{\beta_1\,\beta_2}_{\alpha_1\, \alpha_2} \subset Q^{\beta_1}_{\alpha_1}\,$, where $\,\alpha_2, \beta_2 \in \{+\,,\,-\}\,$.  They form the family $\,\mathcal F_2\,$ of second generation.  More generally, given the family $\,\mathcal F_n\,$ of cornersquares $\, Q^{\beta_1\,\beta_2...\beta_n}_{\alpha_1\, \alpha_2... \alpha_n}  \subset Q^{\beta_1\,\beta_2...\beta_{n-1}}_{\alpha_1\, \alpha_2... \alpha_{n-1}} \in \mathcal F_{n-1}\,$, we take $\,\varepsilon _{n+1}\,$ as the $\,\varepsilon\,$-parameter and adopt to the family $\,\mathcal F_{n+1}\,$ the $\,\varepsilon\,$-cornersquares of $\,Q^{\beta_1\,\beta_2...\beta_n}_{\alpha_1\, \alpha_2... \alpha_n} \,$; namely,
$$\, Q^{\beta_1\,\beta_2...\beta_n +}_{\alpha_1\, \alpha_2... \alpha_n +} \;,\;       Q^{\beta_1\,\beta_2...\beta_n +}_{\alpha_1\, \alpha_2... \alpha_n -} \;,\;       Q^{\beta_1\,\beta_2...\beta_n -}_{\alpha_1\, \alpha_2... \alpha_n -} \;,\;      \, Q^{\beta_1\,\beta_2...\beta_n -}_{\alpha_1\, \alpha_2... \alpha_n +} \;\in\; \mathcal F_{n+1} \, .$$
Thus $\,\mathcal F_{n+1}\,$ consists of  $\,4^{n+1}\,$ cornersquares denoted by  $\,Q^{\beta_1\,\beta_2...\beta_n \beta{n+1}}_{\alpha_1\, \alpha_2... \alpha_n \alpha_{n+1}}\,.$
This process continuous indefinitely.
\subsubsection{The size of squares in $\,\mathcal F_n\,$ and their total area} Let us compare the side-length of squares in $\,\mathcal F_{n+1} \,$ with those in $\,\mathcal F_n\,$.
Every member of $\,\mathcal F_{n+1}\,$ is a cornersquare of a $\,Q \in \mathcal F_n\,$ via the parameter $\,\varepsilon = \varepsilon_{n+1}\,$. Let $\,\ell\,$ denote the sidelength of $\,Q\,$.  We remove from $\,Q\,$ its centersquare  $\, \varepsilon Q\,$. Thus each of the remaining four cornersquares has side-length $\,\frac{1}{2} (1-\epsilon) \ell\,$. For $\,n=1\,$ this equals $\,\frac{1}{2} (1-\epsilon_1)\,$. Hence, by induction, the side-length of squares in $\,\mathcal F_n\,$ equals $\,\frac{1}{2^n} (1 - \varepsilon_1) (1 - \varepsilon_2) \cdots (1 - \varepsilon_n)\, < \frac{1}{2^n}\,$. We have $\,4^n\,$ such squares. This sums up  to the total area of the union
$$
\left|\,\bigcup \mathcal F_n \, \right| =  (1 - \varepsilon_1)^2 (1 - \varepsilon_2)^2 \cdots (1 - \varepsilon_n)^2
$$

\subsubsection{The Cantor set} We have a decreasing sequence of compact sets $\,\bigcup \mathcal F_1 \supsetneq \bigcup\mathcal F_2 \supsetneq ... \supsetneq \bigcup\mathcal F_n ...\,$. Cantor's Theorem tells us that their intersection is not empty,
$$
\mathcal C \bydef \bigcap_{n-1}^\infty  \left( \bigcup \mathcal F_n \right ) \;\not = \emptyset
$$
The measure of this Cantor set is positive.
\begin{equation}\label{eq:cantormeasure}
\left |\, \mathcal C \,\right | \;=\; \lim_{n \rightarrow \infty} \left|\,\bigcup \mathcal F_n \, \right|  = \prod_{k=1}^{\infty} ( 1 - \varepsilon _k)^2  \, > 0
\end{equation}
The latter inequality is a consequence of $\,\sum_{k=1}^\infty\, \varepsilon_k  \,< \infty \,$. 
Every point in $\,\mathcal C\,$ is obtained as intersection of exactly one decreasing sequence of the form
$$\,Q^{\beta_1}_{\alpha_1} \supsetneq Q^{\beta_1\beta_2}_{\alpha_1\alpha_2} \supsetneq
 \;...\; \supsetneq Q^{\beta_1\beta_2... \beta_n}_{\alpha_1\alpha_2... \alpha_n}
 \;... \,$$    An obvious consequence of this is:
 \begin{lemma}\label{NearCantorSet}
 Every open set that intersects $\,\mathcal C\,$ contains a square, say  $\, Q \in \mathcal F_n$ for sufficiently large $\,n\,$ which, in turn, contains its centersquare $\,\varepsilon_n Q \subset Q\,$.
 \end{lemma}
 The idea behind this lemma is  a monotone mapping $\,h: \mathbb Q \onto\mathbb Q\,$ whose branch set will materialize in the centersquares.

\subsection{A monotone map $\,h: \mathbb Q \onto\mathbb Q\,$ } We let $\,\mathscr G\,$ denote the family of centersquares of all generations. From now on the need will not arise for the explicit dependence on multi-indices in the notation of centersquares.
For every $\,Q\in \mathscr G\,$  we have a monotone map $\,h_Q : Q \onto Q\,$  defined by Formula (\ref{RescalledMapping}) with $\,\Phi\,$ given in Lemma \ref{ModelMapping}. Thus the average $\,\textsf{E}^p_q\,$-energy of $\,h_Q\,$ does not depend on $\,Q\,$  and equals $\,\mathbf E\,$. In particular,
\begin{equation}\label{pNorms}
\int_Q \abs{Dh_Q(x)}^p \,\textnormal d x  \; <  \int_Q \Big( \abs{ Dh_Q(x) }^p \,+   [J_{h_Q}(x)]^{-q} \;\Big)  \textnormal d x  \; = |Q| \,\mathbf E \, . 
\end{equation}
Recall that $\,h_Q\,$ equals the identity map near $\,\partial Q\,$. Now we can define the map $\,h \in \mathscr M^p(\overline{\mathbb Q},\overline{ \mathbb Q})\,$ that is hunted by Example ~\ref{ex:branch}.
\begin{definition} \label{DefinitionOFh} We define the map $\,h : \mathbb Q \onto \mathbb Q\,$  by setting:
\begin{equation} \label{FormulaFORh}  h(x)=
\begin{cases}
h_Q(x), \,\,\,&  \textnormal{whenever } \, x \in Q \in \mathscr G \\
x \,, \, &\textnormal{otherwise.}
\end{cases}
\end{equation}
\end{definition}
Let us subtract the identity map.
\begin{equation} \label{FormulaFORh}  h(x) - x =
\begin{cases}
h_Q(x) - x  \,\bydef f_Q(x) , \,\,\, & \textnormal{whenever } \, x \in Q \in \mathscr G \\
0 \,, \, & \textnormal{otherwise.}
\end{cases}
\end{equation}
One advantage of using this is that $\,f_Q \in \mathscr W^{1,p}_0(Q)\,$. Actually, $\,f_Q\,$ vanishes near $\,\partial Q\,$. We have the infinite series
$$\,h(x) - x = \sum_{\Q \in \mathscr G} f_Q(x) \, , \quad \textnormal{in which}  \quad   \sum_{\Q \in \mathscr G} \int_Q |Df_Q |^p \; < \infty \, . $$
This latter inequality is due to the estimate (\ref{pNorms}). Now comes a general fact (rather folklore)  about Sobolev functions:
\begin{lemma} Let $\,\Omega \subset \mathbb R^n\,$  be a bounded domain and $\,\Omega_i \subset \Omega\,, i = 1, 2, ...\,$ disjoint open subsets. Suppose we are given functions $\,f_i \in \mathscr W^{1,p}_0(\Omega_i)\,$ such that  $\,\sum_{i = 1}^\infty \int_{\Omega_i} |Df_i(x)|^p \,\textnormal d x \,< \infty\,$. Then the function

\begin{equation*} \label{FormulaFORh}  f(x)=
\begin{cases}
f_i(x), \,\,\, & \textnormal{whenever } \, x \in \Omega_i\\
0\,, \, & \textnormal{otherwise}
\end{cases}
\end{equation*}
lies in the space $\,\mathscr W^{1,p}_0(\Omega)\,$.
\end{lemma}

We conclude that $\,h \in \mathscr W^{1,p}(\mathbb Q,\mathbb Q)\,$  with $\,p > 2\,$ and, as such, is continuous on $\,\overline{\mathbb Q}\,$. As regards monotonicity, for each square (continuum)  $\,Q\,\in \mathscr G\,$ the mapping $\,h: Q \onto Q\,$ is monotone and $\,h\,$ is the identity outside those continua. This is enough to conclude that $\,h : \overline{\mathbb Q }\onto\overline{\mathbb Q}\,$ is monotone. We leave the details to the reader.\\

Finally, every point of the Cantor set $\,\mathcal C\,$ belongs to the branch set of $\,h\,$. Indeed, by Lemma \ref{NearCantorSet}, any neighborhood of this point contains a square $\,Q \in \mathscr G\,$ in which $\,h = h_Q\,$ fails to be injective. Thus the branch set $\,\mathcal B_h \,$ contains the Cantor set $\,\mathcal C\,$ and, therefore, has positive measure.  On the other hand, by the very definition,   $\,h(x) \equiv x\,$ on $\mathcal C\,$. Therefore $\,h(\mathcal B_h)\,$ also contains $\,\mathcal C\,$, so $\,h(\mathcal B_h)\,$  has positive measure as well.
\begin{remark}
The branch set $\,\mathcal B_h\,$ consists precisely of the Cantor set $\,\mathcal C\,$ and vertical segments in the centersquares $\,Q \in \mathscr G\,$, which are squeezed to the centers. This makes it clear that $\,h(\mathcal B_h) \, = \mathcal C\,$. It should be noted that the branch set $\mathcal B_h$ can have nearly full measure. This follows from the formula~\eqref{eq:cantormeasure} by letting $\varepsilon_k>0$ arbitrarily small.
\end{remark}
The proof of Example~\ref{ex:branch}  is complete.
\subsection{Greater generality} We are now in  a position to appreciate more general approach to the construction presented above. Let us begin with an arbitrary  bounded discrete set $\,\mathbf G\,$ of points in $\,\mathbb R^2\,$ whose limit set, denoted by $\,\mathbf C\,$, has positive area, see Figure~\ref{fig:limitset}. Clearly, $\,\mathbf G\,$ is necessarily countable. Moreover,  $\,\mathbf C \cup \mathbf G \,$ is a compact subset of a bounded domain $\,\Omega \subset \mathbb R^2\,$. Given a point in $\,\mathbf G\,$ we may (and do) choose a square $\,Q \subset \Omega \setminus \mathbf C\,$ centered at this point and small enough so that the family of all such squares, denoted as before by $\,\mathscr G $,  is disjoint. Analogously to Lemma \ref{NearCantorSet}, every open set that intersects $\,\mathbf C\,$ contains a square in $\mathscr G\,$.

\begin{figure}[h]
\centering
\includegraphics[width=10cm]{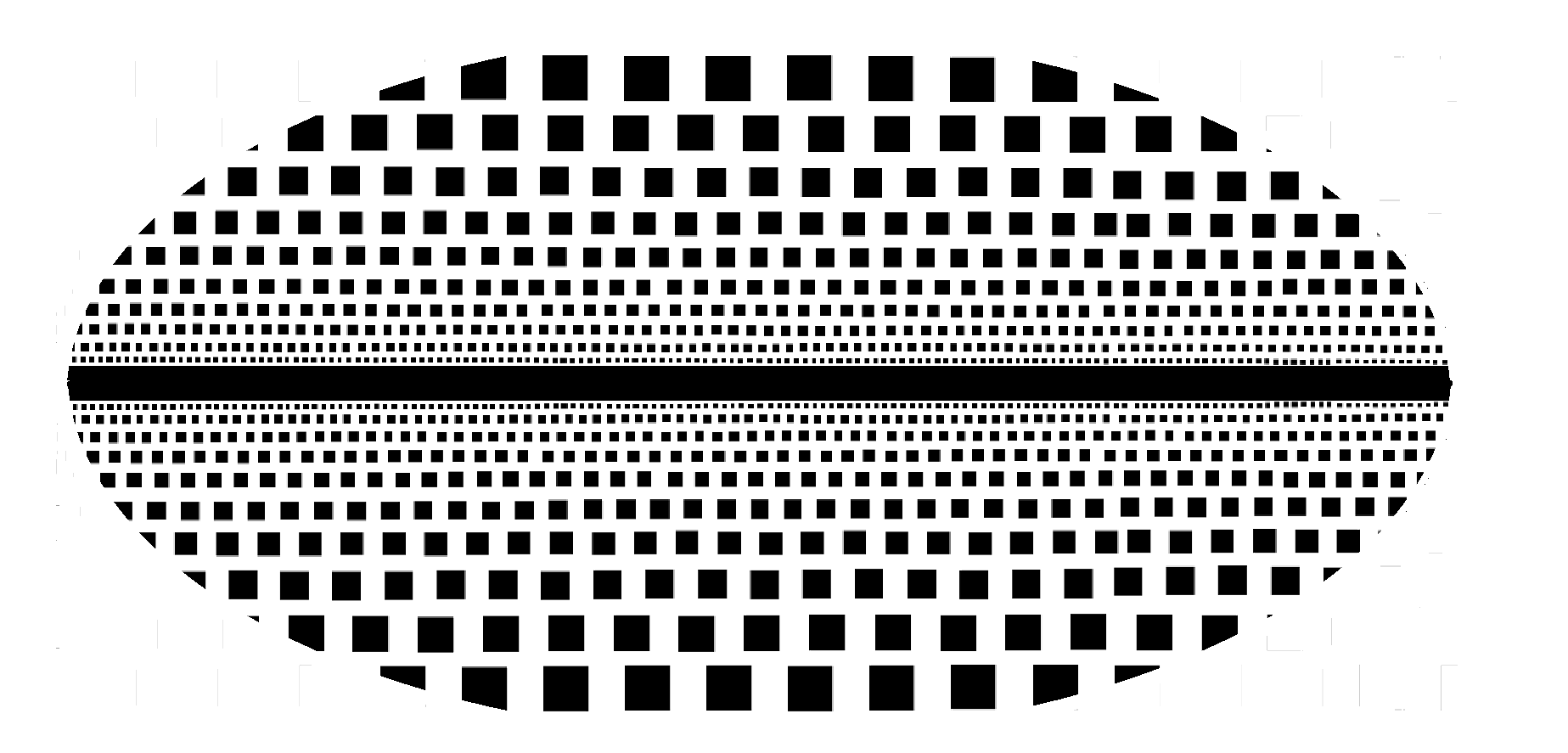}
\caption{Squares approaching the limit set}\label{fig:limitset}
\end{figure}

To every $\,Q \in \mathscr G\,$ there corresponds a monotone map $\,h_Q : Q \onto Q\,$ equal to the identity near  $\,\partial Q\,$. Recall the inequality (\ref{pNorms}),

\begin{equation}\label{pNorms2}
\int_Q |Dh_Q(x)^p \,\textnormal d x  \; <  \int_Q \Big (|Dh_Q(x)^p \,+   [J_{h_Q}(x)]^{-q} \;\Big)  \textnormal d x  \; = |Q| \,\mathbf E \, . 
\end{equation}
This yields
$$
\sum_{Q \in \mathscr G}\int_Q |Dh_Q(x)^p \,\textnormal d x  < \sum_{Q\in \mathscr G} \textsf E^p_q[h_Q]\; = \mathbf E \sum_{Q\in \mathscr G} |Q| \,< \mathbf L |\Omega| < \infty \, . 
$$
It is precisely this property that one needs to infer $\, g \in \mathscr W^{1,p}(\Omega) \cap \mathscr C(\overline{\Omega})\,.$
Exactly the same way as in Formula (\ref{FormulaFORh}), we define a map $\,g : \Omega \onto \Omega\,$ by the rule,
\begin{equation} \label{FormulaFORh2}  g(x)=
\begin{cases}
h_Q(x), \,\,\, &  \textnormal{whenever } \, x \in Q \in \mathscr G \\
x \,, \,  & \textnormal{otherwise.}
\end{cases}
\end{equation}
 Then we conclude in much the same way that $\, g \in \mathscr M^p(\overline{\Omega}, \overline{\Omega})\,$, $\,\textsf{E}^p_q [g] \,< \infty\,$, $\, |\mathcal B_g | > 0\,$ and $\, |g(\mathcal B_g) | > 0\,$.

\section{Proof of Theorem~\ref{thm:monoexistence}}
Let $\X $ and $\Y$ be $\ell$-connected bounded Lipschitz domains in $\R^2$.  Consider a family $\mathcal F$ of Sobolev orientation-preserving monotone mappings $h \colon \overline{\X} \onto \overline{\Y}$ such that
\begin{equation}
\mathsf E^p_q[h] = \int_\X \left(\abs{Dh}^p + \frac{1}{J_h^q}\right) \le E
\end{equation}
for all $h \in \mathcal F$. Here $\X, \Y$, $p\ge 2$, $q>0$ and $E<\infty$ are fixed.
\begin{lemma}\label{lem:unifmodcont}
The family $\mathcal F$ is equicontinuous. Precisely, there is a constant $C$ such that
\begin{equation}\label{eq:modcont}
\abs{h(x_1)-h(x_2)}^2 \le \frac{C \cdot E}{\log \left( 1+ \frac{C}{\abs{x_1-x_2}}\right)}
\end{equation}
for all $h\in \mathcal F$ and distinct points $x_1, x_2 \in \overline{\X}$.
\end{lemma}
\begin{proof} Since $J_h\ge 0$ almost everywhere in $\X$ we have $h \in \mathscr M^p(\overline{\X} , \overline{\Y})$.  Now, for proving~\eqref{eq:modcont} we may  assume (equivalently) that $h\in \mathscr H^p (\X, \Y)$ with $\mathsf E^p_q[h] < \infty$, thanks to Theorem~\ref{thmmono}.  If $\X$ is multiple connected $\ell \ge 2$, then the modulus of continuity estimate~\eqref{eq:modcont} simply follows from the fact that the  Dirichlet energy of $h$ is uniformly bounded by the value of Neohookean energy $E$, see Lemma~\ref{lem:modofcontdirc}. Therefore, it suffices to consider the case of simply connected domains and $p=2$. It is worth recalling that if $\ell =1$, then $\int_\X \abs{Dh}^2 \le E$ is not enough to imply~\eqref{eq:modcont}, see Remark~\ref{rem:modofcontdirc}.

Let $\ell =1$ and $p=2$. We may assume without loss of generality that $\X = \mathbb D = \Y$.  Indeed, for any bounded Lipschitz domain $\Y$ there exists a global bi-Lipschitz change of variables $\Phi \colon \C  \to \C$  for which $\Phi (\Y)$ is the unit disk. Since the finiteness of the $\mathsf E^2_q$-energy is preserved under a bi-Lipschitz change of variables in both the target $\Y$ and the domain $\X$, we assume that  $\X$ and $\Y$ are unit disks. Choose and fix any disk $B=B(x_\circ,\delta) \Subset \X$. We have, for every $h \in \mathcal F$
\[
\begin{split}
\abs{B} & = \int_B J_h^{\frac{q}{q+1}} \cdot \frac{1}{J_h^\frac{q}{q+1}} \\
& \le \left( \int_B J_h \right)^\frac{q}{q+1} \left( \int_B \frac{1}{J^q_h} \right)^\frac{1}{q+1} \le \abs{h(B)}^\frac{q}{q+1} \cdot E^\frac{1}{q+1} \, .
\end{split}
\]
Hence
\[\abs{h(B)} \ge \abs{B}^\frac{q+1}{q} E^{-\frac{1}{q}}  \quad \textnormal{constant independent of } h\in \mathcal F \, . \]
Choose and fix $\varepsilon >0$ such that the annulus
\[\Delta_\varepsilon = \{y \in \Y \colon \dist (y, \partial \Y) \le \varepsilon\}\]
has measure smaller than $\abs{B}^\frac{q+1}{q} E^{-\frac{1}{q}} $. Thus $h(B) \not\subset \Delta_\varepsilon$ and, therefore,  there is a point $a\in B \Subset  \X$ such that $\abs{h(a)}<1 - \varepsilon$. In other words for every $ h \in \mathcal F$ we can find a point  $a \in \X$, with $\abs{a}\le 1-\delta$, and $b=h(a)\in \Y$, with $\abs{b}<1-\varepsilon$. Now consider conformal mappings $\varphi \colon \overline{\X} \onto \overline{\X}$, $\varphi (a)=0$, and $\psi \colon \overline{\Y} \to \overline{\Y}$, $\psi(b)=0$. Both mappings are bi-Lipschitz with bi-Lipschitz constants independent of $a$ and $b$. Thus the energy $\psi \circ h \circ \varphi \colon \X \onto \Y$ is controlled from above by that of $h$ uniformly in $\mathcal F$. Therefore, we may (and do) assume that $h(0)=0$. This leads us to the case of a homeomorphism $h \colon \mathbb D \setminus \{0\} \onto \mathbb D \setminus \{0\}  $; that is, between doubly connected domains. Finally, the inequality~\eqref{eq:modcont}  follows from Lemma~\ref{lem:modofcontdirc}, completing the proof of Lemma~\ref{lem:unifmodcont}.
\end{proof}

\begin{proof}[Proof of Theorem~\ref{thm:monoexistence}]  We apply the direct method in the calculus of variations. For that we take a minimizing sequence $h_k\in \mathscr M^p (\overline{\X}, \overline{\Y})$ of  the neohookean energy $\mathsf E_q^p$ which  converges weakly to $h_\circ$ in $\W^{1,p}(\X , \C)$. Note that here we also used our standing assumption that the class of admissible homeomorphisms is nonempty. Therefore, by Lemma~\ref{lem:poly} we have
\begin{equation}\label{eq:miniproof}
 \mathsf E^p_q [h_\circ] \le \liminf_{k\to \infty}  \mathsf E^p_q [h_k] = \underset{h\in \mathscr M^p (\overline{\X}, \overline{\Y})}{\inf} \mathsf E_q^p [h]   \, .
  \end{equation}
Since $\mathsf E_q^p [h_k] \le E < \infty$ for every $k \in \mathbb N$ and $h_k \to h_\circ$ weakly in $\W^{1,p}(\X , \C)$ applying Lemma~\ref{lem:unifmodcont} we see that the sequence $h_k$ also converges uniformly to $h_\circ$ in $\overline{\X}$. Now the mapping $h_\circ$, being a uniform limit  of monotone mappings $h_k \colon \overline{\X} \onto \overline{\Y}$, is a monotone map from $\overline{\X} $ onto $\overline{\Y} $. Therefore, $h_\circ \in \mathscr M^p (\overline{\X}, \overline{\Y})$. Combining this with~\eqref{eq:miniproof} we have
\[  \mathsf E^p_q [h_\circ] \le \underset{h\in \mathscr M^p (\overline{\X}, \overline{\Y})}{\inf} \mathsf E_q^p [h] \le \mathsf E^p_q [h_\circ]  \, ,  \]
finishing the proof of Theorem~\ref{thm:monoexistence}.
\end{proof}

\section{Proof of Theorem~\ref{thm:monodiscrete}}
\begin{proof}
First we are going to estimate the distortion function
\[ 1 \, \leqslant \, \mathscr K_h(x) \, \bydef \, \frac{|Dh(x)|^2}{2 J_h(x)}\]
by using Young's inequality:
\[ ABC \, \leqslant \, \alpha A^{\frac{1}{\alpha}} + \beta B^{\frac{1}{\beta}} + \gamma C^{\frac{1}{\gamma}}, \,\,\,\, A,B,C \geqslant 0; \, \alpha, \beta, \gamma \geqslant 0, \, \alpha+ \beta+ \gamma=1 \]
where we add here to the convention that  $\gamma C^{\frac{1}{\gamma}}=0$ for $\gamma=0$. The pointwise estimate of $\mathscr K_h$ by means of the energy integrand reads as when $p>2$.
\begin{eqnarray*}
2 \mathscr K_h(x)\, &\leqslant& \, \frac{2}{p} |Dh|^p + \frac{1}{q} \frac{1}{J_h^q} + \left( 1-\frac{2}{p} -\frac{1}{q} \right) \cdot 1\cr \cr
& \leqslant & \frac{2}{p} |Dh|^p + \frac{1}{q} \frac{1}{J_h^q} + \left( 1-\frac{2}{p} -\frac{1}{q}\right) \mathscr K_h
\end{eqnarray*}
Hence \[ \mathscr K_h \, \leqslant \, |Dh|^p +  \frac{1}{J_h^q}= \mathsf E(|Dh|, \det Dh)\]
Integrating over $\Omega$ we obtain \[
\|\mathscr K_h \|_{\mathscr L^{1}(\Omega)} \, = \, \int_{\Omega} \mathscr K_h (x) \dtext x \,  < \infty\]
For Remark~\ref{rem:16} concerning  the case $p=2$ and $q=\infty$ we argue as follows.
\[ \mathscr K_h\,\,=\,\, \frac{|Dh|^2}{2J_h}\,\, \leqslant \,\, \frac{C}{2}|Dh|^2\]
Hence
\[ \int_{\mathbb X} \mathscr K_h \,\,<\,\, \frac{C}{2} \int_{\mathbb X}|Dh|^2\,\, <\,\, \infty\]
In either case, $p>2$ or $p=2$ we see that the map $h \in \mathscr W^{1,2} (\Omega, \mathbb R^2)$ has integrable distortion.
It is known~\cite{IS} that such mappings $h \colon \Omega \rightarrow \mathbb R^2$ are discrete and open. In particular $h(\Omega)$, being an open subset of $\overline{\mathbb Y}$, is contained in $\mathbb Y$.
Next we show that $h \colon \Omega \onto h(\Omega)$ is injective. To this effect suppose, to the contrary, that $h(x_1)=h(x_2)\bydef y_\circ \in h(\Omega)$ for some points $x_1 \neq x_2$ in $\Omega$. The preimage $h^{-1}(y_\circ)$ under the map $h \colon \overline{\mathbb X} \onto \overline{\mathbb Y}$ is a continuum in $\overline{\mathbb X}$ which contains $x_1$, $x_2 \in \Omega$. This contradicts discreteness of $h \colon \Omega \rightarrow \mathbb R^2$.

Since $\Omega \Subset \mathbb X$ is arbitrary it follows that $h \colon \mathbb X \onto h (\mathbb X)$ is a homeomorphism as well. Finally, it remains to show that $h(\mathbb X )= \mathbb Y$. Certainly, $h(\mathbb X)$ being an open subset of $\overline{\mathbb Y}$, is contained in $\mathbb Y$.

Suppose there is $y_\circ \in \mathbb Y \setminus h(\mathbb X)$. But $y_\circ \in{\overline{\mathbb Y}}= h(\overline{\mathbb{X}})$, so $y_\circ =h(x_\circ)$ for some $x_\circ \in \overline{\mathbb X}$. On the other hand, the map $h \colon \overline{\mathbb X} \onto \overline{\mathbb Y}$, being monotone, takes $\partial \mathbb X$ onto $\partial \mathbb Y$.
\[ h(\partial \mathbb X)\,= \,  \partial \mathbb Y \]
This means that $x_\circ \notin \partial \mathbb X$, because $h(x_\circ)= y_\circ \notin \partial \mathbb Y$.\\
In conclusion,
\[ h( \mathbb X) =  \mathbb Y.\]
 \end{proof}

\section{Proof of Theorem~\ref{thm:hopfhomeo}}

 \begin{proof}
  First note that $s >p$ and $|Dh|^s \in \mathscr L_{\loc}^1(\mathbb X)$. The idea of the proof is to infer from ~\eqref{eq:innerp} that
  \begin{equation}\label{eq:claimlastpf}
  \frac{1}{J_h} \in \mathscr L_{\loc}^{\frac{sq}{p}}(\mathbb X) \, . 
  \end{equation}
  For this purpose consider the functions
 \begin{equation}\label{phi}
 \Phi = \left( 1-\frac{p}{2}\right) |Dh|^p + \frac{1+q}{J_h^q} \in \mathscr L_{\loc}^1(\mathbb X)
 \end{equation}
 \begin{equation}\label{Psi}
 \Psi = 2p |Dh|^{p-2} \overline{h_z} h_{\overline{z}} \in \mathscr L_{\loc}^{r}(\mathbb X, \mathbb C) ,\,\,\, \textnormal{ where } r= \frac{s}{p}>1
 \end{equation}
  The equation ~\eqref{eq:innerp} reads as
  \begin{equation}\label{2star}
  \Phi_{\overline{z}} = \Psi_z, \,\,\, \textnormal{ in the sense of distributions }
  \end{equation}
  We first observe that:
  \begin{lemma}\label{lemma112}
  For every subdomain $\Omega \Subset \mathbb X$ compactly contained in $\mathbb X$ it holds that \[ \Phi \in \mathscr L^r(\Omega) \]
   \end{lemma}
  \begin{proof} of Lemma \ref{lemma112}. Choose and fix a function $\lambda \in C_0^{\infty}(\mathbb X)$ which equals $1$ in a neighborhood of $\overline{\Omega} \subset \mathbb X$. Then consider the expression defined in the entire complex plane by the rule
  \begin{equation}\label{1star}
  \lambda \Phi - S \lambda \Psi
  \end{equation}
  where $S \colon \mathscr L^r(\mathbb C) \rightarrow \mathscr L^r(\mathbb C)$, is the Beurling-Ahlfors transform.
  \[ (S f)(z)= -\frac{1}{\pi} \int_{\mathbb C}\int \frac{f(\xi) \dtext \xi}{(z-\xi)^2}, \,\,\, \textnormal{ for } f \in \mathscr L^{r}(\mathbb C)\]
The following identity is characteristic of Beurling-Ahlfors  transform
\[ \frac{\partial }{\partial \overline{z}}(S f)= \frac{\partial }{\partial z} f \,\, \textnormal{ for } f \in \mathscr L^r(\mathbb C)\]
The complex derivatives are understood in the sense of distributions. We may, and do, apply this identity to $f= \lambda \Psi$. Differentiating the expression \eqref{1star} with respect to $\overline{z}$ yields
\begin{eqnarray*}
\frac{\partial}{\partial \overline{z}} \left[ \lambda \Phi - S \lambda \Psi \right]&=&
\frac{\partial}{\partial \overline{z}} \left[ \lambda \Phi \right]- \frac{\partial}{\partial z} \left[ \lambda \Psi \right] \cr \cr
&=& \lambda_{\overline{z}} \Phi + \lambda \Phi_{\overline{z}} - \lambda_z \Psi- \lambda \Psi_z\cr \cr
&=&\lambda_{\overline{z}} \Phi- \lambda_z \Psi=0
\textnormal{ in a neighborhood of } \overline{\Omega}.   \end{eqnarray*}
Here we used the equation \eqref{2star}  and the fact the $\lambda \equiv 1$ in a neighborhood of $\overline{\Omega}$. Thus, by Weyl's Lemma the function:
\[ H = \lambda \Phi- S \lambda \Psi \]
is holomorphic in a neighborhood of $\overline{\Omega} \subset \mathbb X$ so $H \in \mathscr L^{r}(\Omega)$. In this neighborhood we express $H$ as \[  
H(z)= \Phi(z)+ \frac{1}{\pi} \int_{\mathbb C}\int \frac{\lambda(\xi) \Psi(\xi) \dtext \xi}{(z-\xi)^2} \]
The latter integral, being the Beurling-Ahlfors transform of $\lambda \Psi \in \mathscr L^r(\mathbb C)$, represents a function in $\mathscr L^r(\mathbb C)$. In conclusion $\Phi \in \mathscr L^r(\Omega)$.
    \end{proof}
    Now it follows by \eqref{phi} that $J_h^{-q} \in \mathscr L^{r}_{\loc}(\Omega)$, as claimed in~\eqref{eq:claimlastpf}. Combining this with  the assumption $|Dh|^s \in \mathscr L_{\loc}^1(\mathbb X)$ we have 
     \[ \int_{\Omega} \left( |Dh|^s+ {J_h^{-\frac{sq}{p}}}\right) < \infty \, , \]
   where $\Omega \Subset \X$, $s>2$ and $\frac{sq}{p} \geqslant \frac{s}{s-2}$.
  We are now in a position to apply Theorem~\ref{thm:monodiscrete} for $s$ in place of $p$ and $\frac{sq}{p}$ in place of $q$.  Therefore,
    \[ h \colon \mathbb X \onto \mathbb Y \]
    is a homeomorphism.
    \end{proof}

\end{document}